\documentclass[11pt]{article}
\usepackage{amsmath,amssymb,amsthm,amscd}
\usepackage{xcolor}
\usepackage[all]{xy}
\newtheorem{theorem}{Theorem}[section]
\newtheorem{lemma}[theorem]{Lemma}
\newtheorem{proposition}[theorem]{Proposition}
\newtheorem{corollary}[theorem]{Corollary}

\theoremstyle{plain}

\theoremstyle{definition}
\newtheorem{definition}[theorem]{Definition}
\newtheorem{remark}[theorem]{Remark}

\numberwithin{equation}{section}

\renewcommand{\theenumi}{(\roman{enumi})}
\renewcommand{\labelenumi}{\textup{(\theenumi)}}

\title{On the homotopy groups of the automorphism groups of Cuntz--Krieger algebras \\
}
\author{Kengo Matsumoto \\
Department of Mathematics \\
Joetsu University of Education \\
Joetsu, Niigata 943-8512, Japan
\and
Taro Sogabe \\
Department of Mathematics \\
Kyoto University \\
Kyoto,  606-8502, Japan}

\begin{document}


\maketitle

\date{}

\def\det{{{\operatorname{det}}}}

\begin{abstract}
In this paper, we first present the homotopy groups of the automorphism groups 
of Cuntz--Krieger algebras
in terms of the underlying matrices of the Cuntz--Krieger algebras.
We also show that 
the homotopy groups  
are complete invariants 
of the isomorphism class of the Cuntz--Krieger algebras.
As a result,  the isomorphism type of  Cuntz--Krieger algebras 
are completely characterized by 
the group structure of the weak extension groups and the strong extension groups.
\end{abstract}





\newcommand{\Ker}{\operatorname{Ker}}
\newcommand{\sgn}{\operatorname{sgn}}
\newcommand{\Ad}{\operatorname{Ad}}
\newcommand{\ad}{\operatorname{ad}}
\newcommand{\orb}{\operatorname{orb}}
\newcommand{\rank}{\operatorname{rank}}

\def\Re{{\operatorname{Re}}}
\def\det{{{\operatorname{det}}}}
\newcommand{\K}{\operatorname{K}}
\newcommand{\bbK}{\mathbb{K}}
\newcommand{\N}{\mathbb{N}}
\newcommand{\bbC}{\mathbb{C}}
\newcommand{\R}{\mathbb{R}}
\newcommand{\Rp}{{\mathbb{R}}^*_+}
\newcommand{\T}{\mathcal{T}}

\newcommand{\sqK}{\operatorname{K}\!\operatorname{K}}

\newcommand{\Z}{\mathbb{Z}}
\newcommand{\Zp}{{\mathbb{Z}}_+}
\def\AF{{{\operatorname{AF}}}}

\def\Tor{{{\operatorname{Tor}}^\Z_1}}
\def\Ext{{{\operatorname{Ext}}}}
\def\Exts{\operatorname{Ext}_{\operatorname{s}}}
\def\Extw{\operatorname{Ext}_{\operatorname{w}}}
\def\Ext{{{\operatorname{Ext}}}}

\def\OA{{{\mathcal{O}}_A}}
\def\ON{{{\mathcal{O}}_N}}
\def\OAT{{{\mathcal{O}}_{A^t}}}

\def\TA{{{\mathcal{T}}_A}}
\def\TN{{{\mathcal{T}}_N}}

\def\TAT{{{\mathcal{T}}_{A^t}}}

\def\TB{{{\mathcal{T}}_B}}
\def\TBT{{{\mathcal{T}}_{B^t}}}

\def\A{{\mathcal{A}}}
\def\B{{\mathcal{B}}}
\def\C{{\mathcal{C}}}
\def\OaA{{{\mathcal{O}}^a_A}}
\def\OB{{{\mathcal{O}}_B}}
\def\OTA{{{\mathcal{O}}_{\tilde{A}}}}
\def\F{{\mathcal{F}}}
\def\G{{\mathcal{G}}}
\def\FA{{{\mathcal{F}}_A}}
\def\PA{{{\mathcal{P}}_A}}
\def\bbC{{\mathbb{C}}}

 \def\U{{\mathcal{U}}}
\def\OF{{{\mathcal{O}}_F}}
\def\DF{{{\mathcal{D}}_F}}
\def\FB{{{\mathcal{F}}_B}}
\def\DA{{{\mathcal{D}}_A}}
\def\DB{{{\mathcal{D}}_B}}
\def\DZ{{{\mathcal{D}}_Z}}

\def\End{{{\operatorname{End}}}}

\def\Ext{{{\operatorname{Ext}}}}
\def\Hom{{{\operatorname{Hom}}}}

\def\Max{{{\operatorname{Max}}}}
\def\Max{{{\operatorname{Max}}}}
\def\max{{{\operatorname{max}}}}
\def\KMS{{{\operatorname{KMS}}}}
\def\Per{{{\operatorname{Per}}}}
\def\Out{{{\operatorname{Out}}}}
\def\Aut{{{\operatorname{Aut}}}}
\def\Ad{{{\operatorname{Ad}}}}
\def\Inn{{{\operatorname{Inn}}}}
\def\Int{{{\operatorname{Int}}}}
\def\det{{{\operatorname{det}}}}
\def\exp{{{\operatorname{exp}}}}
\def\nep{{{\operatorname{nep}}}}
\def\sgn{{{\operatorname{sign}}}}
\def\cobdy{{{\operatorname{cobdy}}}}
\def\Ker{{{\operatorname{Ker}}}}
\def\Coker{{{\operatorname{Coker}}}}
\def\Im{{\operatorname{Im}}}

\def\ind{{{\operatorname{ind}}}}
\def\Ind{{{\operatorname{Ind}}}}
\def\id{{{\operatorname{id}}}}
\def\supp{{{\operatorname{supp}}}}
\def\co{{{\operatorname{co}}}}
\def\scoe{{{\operatorname{scoe}}}}
\def\coe{{{\operatorname{coe}}}}
\def\I{{\mathcal{I}}}
\def\Span{{{\operatorname{Span}}}}
\def\event{{{\operatorname{event}}}}
\def\S{\mathcal{S}}

\def\coe{{{\operatorname{coe}}}}
\def\scoe{{{\operatorname{scoe}}}}
\def\uoe{{{\operatorname{uoe}}}}
\def\ucoe{{{\operatorname{ucoe}}}}
\def\event{{{\operatorname{event}}}}

\section{Introduction}

Study of homotopy groups of endomorphisms of Cuntz--Krieger algebras 
was initiated by J. Cuntz in \cite{Cuntz1984}.
He showed that the homotopy groups
$\pi_n(\End(\OA\otimes K(H)))$ of the endomorphisms of stabilized Cuntz--Krieger algebra
$\OA\otimes K(H)$ is isomorphic to the bi-variant $n$th weak extension group  $\Extw^n(\OA\otimes K(H),\OA\otimes K(H))$.
By using $\sqK$-theoretic machinery, M. Dadarlat investigated 
the homotopy groups of endomorphisms and automorphisms of Kirchberg algebras,
and showed that the $n$th homotopy group $\pi_n(\Aut(\A))$ of the automorphism group
$\Aut(\A)$ of a unital Kirchberg algebra $\A$ is isomorphic to the $\sqK$-group
$\sqK^{n+1}(C_\A, \A)$,
 where $C_\A$ is the mapping cone of the unital embedding 
$u_\A: \bbC \rightarrow \A$.   
Related to classification of continuous fields of $C^*$-algebras, 
Izumi--Sogabe \cite{IS} and Sogabe \cite{Sogabe2020} studied
the homotopy groups of automorphisms of Cuntz algebras 
and Cuntz--Toeplitz algebras (see also \cite{DP}).

In this paper, we will first present the groups
$\pi_i(\Aut(\OA)), i=1,2 $ for the Cuntz--Krieger algebra $\OA$
in terms of  
the four abelian groups
$K_i(\OA), \Exts^i(\OA), i=0,1$
by using a Dadarlat's general result \cite{Dadarlat2007} on the homotopy groups
of the automorphisms of Kirchberg algebras. 
The former $K_i(\OA)$ is the $\K$-groups 
and the latter $\Exts^i(\OA)$ is the strong extension groups
of $\OA$, where the first strong extension group $\Exts^1$ is the usual 
strong extension group $\Exts$ as in \cite{PS}, \cite{PP} and \cite{MaPre2021exts}.
Since 
$\K_i(\OA), \Exts^i(\OA), i=0,1$
are written in terms of the underlying matrix $A$
(\cite{Cuntz80}, \cite{MaPre2021exts}),
it is a easy task to compute  
the groups $\pi_i(\Aut(\OA)), i=1,2 $ by hand.

In \cite{Sogabe2022}, the second named author introduced the notion of reciprocality for
two unital Kirchberg algebras with finitely generated $\K$-groups.
Two unital Kirchberg algebras $\A, \B$ are said to be reciprocal if both
$\A \underset{\sqK}{\sim}D(C_\B)$ and
$\B \underset{\sqK}{\sim}D(C_\A)$
hold, 
where $\underset{\sqK}{\sim}$ means $\sqK$-equivalence, 
and 
the $C^*$-algebra  
$D(C_\A)$ is the Spanier--Whitehead $\K$-dual defined by Kaminker--Schochet
\cite{KS} (cf. \cite{KP}). 
The Spanier--Whitehead $\K$-duality is a noncommutative analogue of the classical 
 Spanier--Whitehead duality for finite CW complexes.
Related to the homotopy groups of automorphisms of unital Kirchberg algebras,
the second named author proved in \cite{Sogabe2022} that  
 $\pi_i(\Aut(\A))$ is isomorphic to $\pi_i(\Aut(\B))$ for $i=1,2$  
 if and only if $\A$ is isomorphic to $\B$, or $\A$ and  $\B$ are reciprocal.
 In the present paper, we will show that every pair $\OA$ and $\OB$ of 
 two Cuntz--Krieger algebras can never be reciprocal
 by computing the $\K$-groups of $D(C_\OA)$ and $D(C_\OB)$
  (Lemma \ref{lem:nonreciprocal}).
In particular, the groups $\pi_i(\Aut(\OA))$ and $\pi_i(\Aut(\OB))$ for $i=1,2$ are isomorphic 
 if and only if $\OA$ is isomorphic to $\OB$.   
Since we know that the group structure 
of the two groups
$\K_0(\OA), \Exts^1(\OA)$ determine the other two groups
 $\K_1(\OA), \Exts^0(\OA)$ (Lemma \ref{lem:K0Ext1=>K_1Ext0}),
 the homotopy groups 
  $\pi_i(\Aut(\A)), i=1,2$ 
 are determined by the only two groups $\K_0(\OA), \Exts^1(\OA)$. 
For an $N \times N$ matrix $A$ with entries in $\{0,1\}$,
let us denote by $\widehat{A}$ 
the  $N\times N$ matrix 
$ \widehat{A} = A + R_1 - AR_1$, 
where $R_1$ is the $N\times N$ matrix such that its first row is 
the vector $[1,\dots,1]$ whose entries are all $1$s
and the other rows are zero vectors.  
  As a result, we will prove the following theorem in the present paper.
\begin{theorem}[{Theorem \ref{thm:mainthm}}]
Let $A=[A(i,j)]_{i,j=1}^N, B=[B(i,j)]_{i,j=1}^M$ 
be irreducible non permutation matrices with entries in $\{0,1\}.$
Then the following four conditions are mutually equivalent.
 \begin{enumerate}
\renewcommand{\theenumi}{(\roman{enumi})}
\renewcommand{\labelenumi}{\textup{\theenumi}}
\item
$\pi_i(\Aut(\OA))\cong \pi_i(\Aut(\OB)), i=1,2$.
\item $\OA \cong \OB.$
\item
$\K_0(\OA)\cong \K_0(\OB)$ and $\Exts(\OA) \cong \Exts(\OB)$.
\item
$\Z^N/(I-A) \Z^N \cong \Z^M/(I-B) \Z^M$ and 
$\Z^N/(I-\widehat{A}) \Z^N \cong \Z^M/(I-\widehat{B}) \Z^M$.
\end{enumerate}
\end{theorem}
Since 
$\K_0(\OA)$ is isomorphic to  the weak extension group $\Extw(\OA)$,
the condition (iii) is replaced with
$$
\Extw(\OA)\cong \Extw(\OB) \quad \text{ and } \quad
\Exts(\OA) \cong \Exts(\OB).
$$
The homotopy groups of $\Aut(\OA)$ 
completely determine the isomorphism class of $\OA$,
and also the isomorphism class of $\OA$
is determined by the group structure of 
the two extension groups $\Extw(\OA)$ and $\Exts(\OA)$,
which are computed as the abelian groups 
$\Z^N /(I-{A})\Z^N$ and $\Z^N /(I- \widehat{A})\Z^N$  
that are the cokernels of the matrices $I- {A}$ by \cite{CK}
and $I- \widehat{A}$ by \cite{MaPre2021exts}, respectively.
Thus, our theorem implies that the pair of two extension groups has exactly the same information as the pair $(\K_0(\OA), [1_\OA]_0)$ due to \cite{Ro}.
The possible range of the pair $(\Extw(\OA), \Exts(\OA))$ is clarified in Corollary \ref{range},
and the relationship between $(\Extw(\OA), \Exts(\OA))$ and $(\K_0(\OA), [1_\OA]_0)$ is discussed in Proposition \ref{rel}.

Concerning the homotopy groups of the automorphism group of the stabilizations
$\OA\otimes K(H)$ together with Dadarlat's result in \cite{Dadarlat2007},
we will show that
$\pi_i(\Aut(\OA\otimes K(H)))\cong \pi_i(\Aut(\OB\otimes K(H))), i=1,2$
if and only if $\OA\otimes K(H) \cong \OB\otimes K(H)$
(Proposition \ref{prop:stabilization2}).

\section{Preliminary}
Throughout the paper, we mean by a Kirchberg algebra 
a separable unital nuclear simple purely infinite $C^*$-algebra.
We always assume that a Kirchberg algebra satisfies the Universal Coefficient Theorem (UCT). 
\subsection{Extension groups as $\sqK$-groups}
Let $\A$ be a unital Kirchberg algebra.
Let $u_\A: \bbC\rightarrow \A$ be the unital embedding
defined by $u_\A(c) = c 1_\A$ for $c \in \bbC$,
where $1_\A$ denotes the unit of $\A$.
The mapping cone $C_\A$ for the map $u_\A: \bbC\rightarrow \A$
is defined by the $C^*$-algebra 
\begin{equation*}
C_\A: = \{ f \in C_0(0,1]\otimes\A \mid f(1) \in \bbC 1_\A\}.
\end{equation*}
The suspension $S\A$ is the $C^*$-algebra $C_0(0,1)\otimes\A$
which is naturally embedded into $C_\A$.
We have then a short exact sequence 
\begin{equation}\label{eq:SACA}
0 \longrightarrow S\A \longrightarrow C_\A \longrightarrow \bbC \longrightarrow 0
\end{equation}
in a natural way.
For a separable unital nuclear $C^*$-algebras $\A, \B$, we write
\begin{align*}
\Exts^n(\A, \B) =  \sqK^{1-n}(C_\A, \B), \qquad 
\Extw^n(\A, \B) =  \sqK^{n}(\A, \B), \qquad n=0,1,
\end{align*}
where $\sqK^n(\,\,\, , \,\,\,) $ means the Kasparov $\sqK$-group of degree $n$
(\cite{Kasparov81}, cf. \cite{Blackadar}).
For separable UCT C*-algebras $\A, \B$,
one can compute the KK-groups by the following short exact sequence, called UCT, splitting unnaturally (see \cite{Blackadar, Brown84, RS}):
\[0\to\bigoplus_{i=0, 1}\Ext_\Z^1(\K_i(\A), \K_{i+1}(\B))\to KK(\A, \B)\to \bigoplus_{i=0, 1}\Hom (\K_i(\A), \K_i(\B))\to 0.\]

We in particular write 
\begin{align*}
\Exts^n(\A) =  \Exts^n(\A, \bbC), \qquad 
\Extw^n(\A) =  \Extw^{n}(\A, \bbC), \qquad n=0,1.
\end{align*}
The $\K$-homology groups $\K^n(\A), n=0,1$ are defined by $\sqK^n(\A,\bbC)$,
so that 
\begin{align*}
\Exts^n(\A) =  \K^{1-n}(C_\A), \qquad 
\Extw^n(\A) =  \K^{n}(\A), \qquad n=0,1
\end{align*}
(see \cite{Blackadar}, \cite{HR}, \cite{Skandalis}  for detail accounts on the relation of 
Ext-groups and $\sqK$-theory).  
The extension groups $\Exts^n(\A), \Extw^n(\A)$ were primary defined and studied 
to investigate extensions of $C^*$-algebras motivated 
by classification of essentially normal 
operators on Hilbert spaces (cf. \cite{BDF}, \cite{Douglas}, etc.).
Let $B(H)$ denote the $C^*$-algebra of bounded linear operators 
on a separable infinite dimensional Hilbert space $H$.
The quotient $C^*$-algebra $Q(H)$ of $B(H)$ by 
the $C^*$-subalgebra $K(H)$ of compact operators on $H$ is called the Calkin algebra.   
For a separable unital nuclear $C^*$-algebra $\A$,
a $*$-homomorphism $\tau: \A\rightarrow Q(H)$ is called an extension of $\A$.
The extension is said to be unital (resp. essential) if $\tau$ is unital (resp. injective).
Two extensions  $\tau_1, \tau_2: \A\rightarrow Q(H)$
are said to be strongly (resp. weakly) equivalent if there exists a unitary 
$U \in B(H)$ (resp. $u \in Q(H))$ 
such that 
$\tau_2(a) = \pi(U)\tau_1(a) \pi(U)^*, a \in \A$
(resp. $\tau_2(a) = u\tau_1(a)u^*, a \in \A$).
Let us denote by 
$\Exts(\A)$ (resp. $\Extw(\A)$) the set of strong (resp. weak) equivalence classes of  unital essential extensions
of $\A$.
It is well-known that both $\Exts(\A)$ and $\Extw(\A)$
 become abelian groups whose addition is defined 
 by direct sum of extensions (cf. \cite{BDF}, \cite{Douglas},
 \cite{HR}, etc.).
It is also well-known that
$\Exts(\A)$ and $\Extw(\A)$ are isomorphic to $\Exts^1(\A)$ and $\Extw^1(\A)$
as abelian groups, respectively (cf. \cite[2.4 Corollary]{Skandalis}, \cite[Theorem 4.5]{PennigSogabe}).  
By a general theory of $\sqK$-theory, the short exact sequence \eqref{eq:SACA}
of $C^*$-algebras yield the following cyclic six term exact sequence (see \cite{MaPre2021exts, Skandalis})
\begin{equation} \label{eq:Extcyclicexact}
0\to \Exts^0(\A)\to\Extw^0(\A)\to\Z\to\Exts^1(\A)\to\Extw^1(\A)\to 0.
\end{equation}
\subsection{The cyclic exact sequence for Cuntz--Krieger algebras}
In what follows, we denote by $A$ an $N\times N$ irreducible non permutation matrix 
$[A(i,j)]_{i,j=1}^N$ with entries in $\{0,1\}.$
The Cuntz--Krieger algebra $\OA$ is defined to be the universal $C^*$-algebra 
generated by $N$-partial isometries $S_1, \dots, S_N$ subject to
 the operator relations:
\begin{equation*}
S_i^* S_i = \sum_{j=1}^n A(i,j) S_j S_j^*, \quad i=1,\dots, N, \qquad \sum_{j=1}^N S_j S_j^* =1. 
\end{equation*}  
Let $\hat{\iota}_A: \Z \rightarrow \Z^N/ (I-\widehat{A})\Z^N$
be the homomorphism of abelian groups defined by
$\hat{\iota}_A(m)= [(I-A)[k_i]_{i=1}^N]$ where $\sum_{i=1}^N k_i =m$ with $k_i \in \Z$ (see \cite{MaPre2021exts}).
Since $I-\widehat{A}=(I-A)(I-R_1)$, the map $\hat{\iota}_A$ is well-defined.
Let us denote by $\Ker(I-A), \Ker(I - \widehat{A})$ 
the subgroups of $\Z^N$ defined by the kernels of the homomorphisms
$I - A: \Z^N \rightarrow \Z^N, \, I - \widehat{A}: \Z^N \rightarrow \Z^N,$
respectively.
In \cite{MaJMAA2024}, the cyclic six term exact sequence \eqref{eq:Extcyclicexact} for $\A=\OA$
is computed  as follows
\begin{equation}\label{eq:6termA}
0\to \Ker(I-\widehat{A})/i_1(\Z)\xrightarrow{j_A}\Ker(I-A)\xrightarrow{s_A}\Z
\xrightarrow{\hat{\iota}_A}\Z^N/(I-\widehat{A})\Z^N\xrightarrow{\hat{q}_A}\Z^N/(I-A)\Z^N\to 0.
\end{equation}
where $i_1:\Z\rightarrow \Ker(I-\widehat{A})$ and the two maps $s_A, j_A$ are defined by 
\[i_1(m) :=
\left[
\begin{array}{c}
m\\
0\\
\vdots\\
0
\end{array}
\right]
,\quad s_A(
[l_i]_{i=1}^N ) 
:=
\sum_{i=1}^N l_i, \]

\[j_A(
[l_i]_{i=1}^N+i_1(\mathbb{Z}))
:=(I-R_1)([l_i]_{i=1}^N)
=
\left[
\begin{array}{c}
- \sum_{i=2}^N l_i\\
 l_2\\
\vdots\\
l_N
\end{array}
\right],\] 
respectively, and 
$\hat{q}_A: \Z^N/ (I-\widehat{A})\Z^N\rightarrow \Z^N/ (I-{A})\Z^N$
denotes the natural quotient map.

Let us consider the Toeplitz algebra $\TA$ for the matrix $A$, 
which is defined by 
 the universal $C^*$-algebra 
generated by $N$-partial isometries $T_1, \dots, T_N$ and one non zero projection $P_0$ 
subject to
 the operator relations (see \cite{EFW, EV}):
\begin{equation*}
T_i^* T_i = \sum_{j=1}^n A(i,j) T_j T_j^* + P_0, \quad i=1,\dots, N, \qquad \sum_{j=1}^N T_j T_j^*  + P_0=1. 
\end{equation*}  
 The correspondence $T_i \rightarrow S_i, i=1,\dots,N$
 yields a short exact sequence
 \begin{equation}\label{eq:toeplitzext}
 0 \longrightarrow K(H) 
  \longrightarrow \TA 
   \longrightarrow \OA 
   \longrightarrow 0 
 \end{equation}
 called the Toeplitz extension of $\OA$.
 The short exact sequence \eqref{eq:toeplitzext} for the transposed matrix $A^t$
 yields the 
 cyclic six term exact sequence 
\begin{equation}\label{eq:6termTATOA}
0\to\K_1(\TAT)\to \K_1(\OAT)\to \Z\to \K_0(\TAT)\to \K_0(\OAT)\to 0
\end{equation}
 which is nothing but the cyclic six term exact sequence \eqref{eq:6termA}
 (\cite[Proposition 4.6]{MaJMAA2024}).
 We summarize the extension groups and $\K$-groups of Cuntz--Krieger algebras
 in the following way.
 \begin{lemma}[{\cite{Cuntz80}, \cite{CK}, \cite{MaJMAA2024}, \cite{MaPre2021exts}}]
\label{lem:summary}
Let $A$ be an $N\times N$ irreducible non permutation matrices with entries in $\{0,1\}.$ 
Then we have
\begin{align*}
\Extw^1(\OA) = & \Extw(\OA) =   \Z^N/ (I-A)\Z^N  = \K_0(\OAT)\cong \K_0(\OA),\\
\Exts^1(\OA) = & \Exts(\OA) =  \Z^N/ (I-\widehat{A})\Z^N  = \K_0(\TAT),\\
\Extw^0(\OA) = & \Ker(I - A) =   \K_1(\OAT)\cong \K_1(\OA),\\
\Exts^0(\OA) = & \Ker(I - \widehat{A})/ i_1(\Z)
= \Ker(s_A: \Ker(I - A) \rightarrow \Z) =  \K_1(\TAT).
\end{align*}
\end{lemma}
We note that the Smith normal forms for the $\Z$-module maps $I-A, I-A^t$  show that $\Z^N/(I-A)\Z^N\cong \Z^N/(I-A^t)\Z^N=K_0(\OA)$
and $\Ker(I-A)\cong \Ker(I-A^t)=K_1(\OA)$.
\section{Formulas of $\pi_i(\Aut(\OA))$} 
Let us denote by $\Aut(\A)$ the group of automorphisms of a $C^*$-algebra $\A$.
It has a topology defined by pointwise norm convergence.  
The $n$th homotopy group of $\Aut(\A)$ is denoted by $\pi_n(\Aut(\A)), n=1,2.\dots$
The following theorem is due to M. Dadarlat \cite{Dadarlat2007}.
\begin{theorem}[{Dadarlat \cite[Corollary 5.10]{Dadarlat2007}}] \label{thm:Dadarlat}
For a unital Kirchberg algebra $\A$ and 
$ n=1,2,\dots
$,
we have the following formulas 
\[\pi_n(\Aut(\A)) \cong KK^{n+1}(C_\A, \A)=\Exts^n(\A, \A),
\]
\[\pi_n(\Aut(\A\otimes K(H))) \cong KK^{n}(\A, \A)=\Extw^n(\A, \A).
\]
\end{theorem}

\subsection{Formulas of $\pi_i(\Aut(\OA))$}

\begin{lemma}\label{lem:Kunneth}
Let $\A$ be a separable unital nuclear $C^*$-algebra with finitely generated $\K$-groups. 
We have short exact sequences of abelian groups:
\begin{align*}
0& \longrightarrow 
  (\Exts^1(\A)\otimes \K_0(\A)) \, \oplus  \, (\Exts^0(\A)\otimes \K_1(\A)) \\
& \longrightarrow \pi_1(\Aut(\A)) \\
& \longrightarrow 
\Tor(\Exts^1(\A), \K_1(\A)) \, \oplus  \, \Tor(\Exts^0(\A), \K_0(\A)) 
\longrightarrow 0, \\
\intertext{ and }
0& \longrightarrow 
  (\Exts^1(\A)\otimes \K_1(\A)) \, \oplus  \, (\Exts^0(\A)\otimes \K_0(\A)) \\
& \longrightarrow \pi_2(\Aut(\A)) \\
& \longrightarrow 
\Tor(\Exts^1(\A), \K_0(\A)) \, \oplus  \, \Tor(\Exts^0(\A), \K_1(\A)) 
\longrightarrow 0.
\end{align*}
\end{lemma}
\begin{proof}
Since
$\pi_1(\Aut(\A)) \cong \sqK(C_\A, \A)$,
the K{\"u}nneth theorem (see \cite[Theorem 23.1.2]{Blackadar}, \cite{RS}) tells us that
there exists a short exact sequence
\begin{align*}
0& \longrightarrow 
  (\K^0(C_\A)\otimes \K_0(\A)) \, \oplus  \, (\K^1(C_\A)\otimes \K_1(\A)) \\
& \longrightarrow \pi_1(\Aut(\A)) \\
& \longrightarrow 
\Tor(\K^0(C_\A), \K_1(\A)) \, \oplus  \, \Tor(\K^1(C_\A), \K_0(\A)) 
\longrightarrow 0.
\end{align*}
As $\Exts^i(\A) = \K^{i+1}(C_\A), i=0,1$,
we have the desired short exact sequence.

The exact sequence for $\pi_2(\Aut(\A))$ is proved smilarly.
\end{proof}
Let $A=[A(i,j)]_{i,j=1}^N $ be an irreducible $N\times N$ matrix with entries in $\{0,1\}$.
 By applying Lemma \ref{lem:Kunneth} to the Cuntz--Krieger algebra $\OA$, 
 we have the following formulas.
\begin{proposition}\label{prop:pi_1}
\hspace{6cm}
\begin{align*}
\pi_1(\Aut(\OA)) 
\cong &  
(\Exts^1(\OA)\otimes \K_0(\OA)) \, \oplus  \, (\Exts^0(\OA)\otimes \K_1(\OA)), 
\quad\\
\pi_2(\Aut(\OA)) 
\cong & 
(\Exts^1(\OA)\otimes \K_1(\OA)) \, \oplus  \, (\Exts^0(\OA)\otimes \K_0(\OA)) \, 
\, \oplus  \Tor(\Exts^1(\A), \K_0(\A)).
\end{align*}
\end{proposition}
\begin{proof}
Since
$\K_1(\OA) = \Ker(I- A^t)$ and
$\Exts^0(\OA) = \Ker(s_A: \Ker(I-A) \rightarrow \Z)$ are torsion free,
one has
$$
\Tor(\Exts^1(\OA), \K_1(\OA)) = \Tor(\Exts^0(\OA), \K_0(\OA)) = 0,
$$ 
and this shows the desired formula for $\pi_1(\Aut(\OA))$ by Lemma \ref{lem:Kunneth}.
The desired formula for $\pi_2(\Aut(\OA))$ 
is proved similarly. 
\end{proof}
For an $N \times N$ matrix $A$, 
let us define an $(N+1) \times N$ matrix  $A_T$ by
$$
A_T =
\begin{bmatrix}
1 & \cdots & 1 \\
  & I-A    &
\end{bmatrix}.
$$
It is direct to see that 
$\Ker(A_T: \Z^N \rightarrow \Z^{N+1}) = \Ker(s_A: \Ker(I-A)\rightarrow \Z)$
(cf. \cite[Lemma 4.2]{MaJMAA2024}). 
\begin{corollary}\label{cor:pi_1}
\hspace{6cm}
\begin{align*}
\pi_1(\Aut(\OA)) 
\cong &  
(\Z^N/ (I - \widehat{A}) \Z^N \otimes \Z^N/ (I-A^t) \Z^N)  
\, \oplus \,  
(\Ker(A_T: \Z^N \rightarrow \Z^{N+1}) \otimes \Ker(I-A^t)),
\\
\pi_2(\Aut(\OA)) 
\cong & 
(\Z^N/ (I - \widehat{A}) \Z^N \otimes \Ker(I-A^t)) \,
 \oplus \, (\Ker(A_T: \Z^N \rightarrow \Z^{N+1}) \otimes \Z^N/ (I-A^t) \Z^N)   \\
 &\, \oplus  \, 
\Tor(\Z^N/ (I - \widehat{A}) \Z^N, \Z^N/ (I-A^t) \Z^N). 
\end{align*}
\end{corollary}
\begin{remark}
We note that direct computation yields
\[\pi_1(\operatorname{Aut}(\OA))\cong \pi_2(\operatorname{Aut}(\OA))\oplus \operatorname{Tor}(\Z^N/(I-A)\Z^N),\]
where $\operatorname{Tor}(-)$ denotes the torsion part (see examples in Section 6).
\end{remark}

\subsection{$\pi_i(\Aut(\OA\otimes K(H)))$}
Combining Dadarlat's result \[\pi_n(\Aut(\A\otimes K(H))) \cong KK^{n}(\A, \A), \quad n=1,2, \dots\] with UCT and Lemma \ref{lem:summary},
we know the following formulas. 
\begin{proposition}\label{prop:pi_1timesK}
\begin{align*}
     & \pi_1(\Aut(\OA\otimes K(H)) \\ 
\cong & 
(\Extw^1(\OA)\otimes \K_0(\OA)) \, \oplus  \, (\Extw^0(\OA)\otimes \K_1(\OA)) \\
= & 
(\Z^N/ (I - {A}) \Z^N\otimes \Z^N/ (I - {A^t}) \Z^N) 
\, \oplus  \, 
(\Ker(I-A)\otimes \Ker(I-A^t)),\\
 \intertext{ and }
      & \pi_2(\Aut(\OA\otimes K(H))) \\
\cong & 
(\Extw^1(\OA)\otimes \K_1(\OA)) \, \oplus \, (\Extw^0(\OA)\otimes \K_0(\OA)) \\
      & \,  \oplus \,  \Tor(\Extw^1(\OA), \K_0(\OA)) \\
= & (\Z^N/ (I - {A}) \Z^N\otimes \Ker(I-A^t)) \, \oplus \, 
(\Ker(I-A) \otimes  \Z^N/ (I - {A^t}) \Z^N) \\
      & \,  \oplus \, \Tor(\Z^N/ (I - {A}) \Z^N, \Z^N/ (I - {A^t}) \Z^N). 
\end{align*}
\end{proposition}

\section{$\pi_i(\Aut(\OA))$ and reciprocality}
\subsection{Reciprocality in $\K$-duals}
In \cite{Sogabe2022}, 
the second named author introduced the notion of reciprocality for a pair of
unital Kirchberg algebras $\A, \B$ with finitely generated $\K$-groups
to investigate the homotopy groups 
$\pi_i(\Aut(\A))$ of the automorphism group of $\A$
and continuous fields of $C^*$-algebras, so called bundles of $C^*$-algebras.
Suppose that $\C$ is a separable nuclear UCT $C^*$-algebra
 with finitely generated $\K$-groups.
Then Kaminker--Schochet \cite{KS} showed 
that there exists another separable nuclear UCT $C^*$-algebra written $D(\C)$
unique up to $\sqK$-equivalence satisfying the duality 
$$
\sqK^i(\C, \bbC) \cong \sqK^i(\bbC, D(\C)), \qquad i=0,1.
$$
The $C^*$-algebra $D(\C)$ is called the Spanier--Whitehead $\K$-dual of $\C$
(see \cite{KP}, \cite{KS}, etc.).
By Theorem \ref{thm:Dadarlat} due to M. Dadarlat, we see that
for a Kirchberg algebra $\A$ with finitely generated 
$\K$-groups, the mapping cone $C_\A$ has its dual $D(C_\A)$ such that 
\begin{equation} \label{eq:piAutKKDC}
\pi_i(\Aut(\A)) \cong \K_{i+1}(D(C_\A)\otimes \A), \qquad i=1,2,\dots.
\end{equation}
Suppose $\pi_i(\Aut(\A)) \cong \pi_i(\Aut(\B)), i=1,2, \dots.$
By \eqref{eq:piAutKKDC}, we have
$D(C_\A)\otimes \A \underset{\sqK}{\sim} D(C_\B)\otimes \B.$
 The reciprocality for a pair $\A$ and $\B$ of unital Kirchberg algebras
 was introduced in \cite{Sogabe2022}
 in the following way.
\begin{definition}[{Sogabe \cite{Sogabe2022}}]\label{def:reciprocal}
Let $\A, \B$ be unital UCT Kirchberg algebras 
both of them have finitely generated $\K$-groups.
Then $\A$ and $\B$ are said to be reciprocal if both 
$\A \underset{\sqK}{\sim} D(C_\B)$ and
$\B \underset{\sqK}{\sim} D(C_\A)$ hold.
\end{definition}
The second named author proved the following theorem in \cite{Sogabe2022}.
\begin{theorem}[{\cite[Theorem 1.2]{Sogabe2022}}] \label{thm:Sogabe}
The homotopy groups 
$\pi_i(\Aut(\A))$ and $\pi_i(\Aut(\B))$ are isomorphic for all $ i=1,2, \dots$
if and only if 
either $A\cong \B$ or $\A$ and $\B$ are reciprocal.
\end{theorem}
Let us provide a couple of lemmas to study the reciprocality in  Cuntz--Krierger algebras.

\begin{lemma}\label{lem:DCOATAT}
For an irreducible non permutation matrix $A$ with entries in $\{0,1\}$,
we have
$D(C_{\OA}) \underset{\sqK}{\sim} \TAT$,
that is,
the K-dual of the mapping cone $C_{\OA}$ of $\OA$ is $\sqK$-equivalent to the Toeplitz algebra $\TAT$ 
defined by the transposed matrix $A^t$.
\end{lemma}
\begin{proof}
Since $\OA, \OAT$ satisfy UCT,
$C_\OA$ and $\TAT$ also satisfy UCT.
We have 
$$
\sqK(C_\OA, \bbC)  =\Exts^1(\OA) =\Exts(\OA) = \Z^N/(I- \widehat{A})\Z^N = \K_0(\TAT) 
=\sqK(\bbC, \TAT),
$$
$$
\sqK^1(C_\OA, \bbC)  =\Exts^0(\OA)  = \Ker(s_A: \Ker(I-A) \rightarrow \Z)
= \K_1(\TAT) 
=\sqK^1(\bbC, \TAT),
$$ 
and hence $\TAT$ is the Spanier--Whitehead $\K$-dual of $C_{\OA}$.
\end{proof}

For a finitely generated abelian group $G$,
we write  the rank of the torsion free part of $G$
as $\rank(G)$.
The group $G$ is a finitely generated $\Z$-module and 
$\rank(G)$ is the dimension of the $\mathbb{Q}$-vector space $\mathbb{Q}\otimes_\Z G$.
\begin{lemma}\label{lem:rankTA}
For an irreducible matrix $A=[A(i,j)]_{i,j=1}^N$ 
with entries in $\{0,1\},$
we have
\begin{equation*}
\rank(\K_0(\TAT))= \rank(\K_1(\TAT)) +1.
\end{equation*}
\end{lemma}
\begin{proof}
Applying the exact functor $\mathbb{Q}\otimes_\Z-$ to the sequence \eqref{eq:6termTATOA},
one has the following exact sequence of $\mathbb{Q}$-vector spaces
\[0\to \mathbb{Q}^{\rank(\K_1(\TAT))}\to\mathbb{Q}^{\rank(\K_1(\OAT))}\to\mathbb{Q}\to\mathbb{Q}^{\rank(\K_0(\TAT))}\to\mathbb{Q}^{\rank(\K_0(\OAT))}\to 0,\]
which implies $\rank(\K_1(\TAT))-\rank(\K_1(\OAT))+1-\rank(\K_0(\TAT))+\rank(\K_0(\OAT))=0$.
The statement follows from the equation \[\rank(\K_1(\OAT))=\rank (\Ker (I-A))=\rank (\Z^N/(I-A)\Z^N)=\rank(\K_0(\OAT)).\]


\end{proof}
\begin{lemma}\label{lem:nonreciprocal}
For every pair of irreducible non permutation matrices $A, B$ with entries in $\{0,1\}$,
$\OA$ is not reciprocal to $\OB$. 
\end{lemma}
\begin{proof}
Suppose that $\OA$ and  $\OB$ are reciprocal (i.e.,
$\OA\underset{\sqK}{\sim}D(C_{\OB}),
$
$
\OB\underset{\sqK}{\sim}D(C_{\OA})).
$
Since Lemma \ref{lem:DCOATAT} shows 
$
\OA\underset{\sqK}{\sim}\TBT
$,
one has $\K_*(\OA) \cong \K_*(\TBT)$ and $\rank(\K_0(\TBT))=\rank(\K_0(\OA))=\rank(\K_1(\OA))=\rank(\K_1(\TBT))$.
This contradicts Lemma \ref{lem:rankTA}.
\end{proof}
\begin{theorem}\label{thm:homotopy=isom}
The homotopy groups $\pi_i(\Aut(\OA))$ and $\pi_i(\Aut(\OB))$
are isomorphic for every $i=1,2$
if and only if $\OA$ is isomorphic to $\OB.$
\end{theorem}
\begin{proof}
The assertion follows from Theorem \ref{thm:Sogabe} 
together with Lemma \ref{lem:nonreciprocal}.
\end{proof}

\subsection{Main theorem}

\begin{lemma}\label{lem:K0Ext1=>K_1Ext0}
Let $A=[A(i,j)]_{i,j=1}^N, B=[B(i,j)]_{i,j=1}^M$ 
be irreducible non permutation matrices with entries in $\{0,1\}.$
Suppose that 
$\K_0(\OA)\cong \K_0(\OB)$ and $\Exts^1(\OA) \cong \Exts^1(\OB).$
Then we have
$\K_1(\OA)\cong \K_1(\OB)$ and $\Exts^0(\OA) \cong \Exts^0(\OB).$
\end{lemma}
\begin{proof}
Let us denote by $T_A$ (resp. $\tilde{T}_A$) the torsion part of $\K_0(\OAT)\cong \K_0(\OA)$ (resp. $\K_0(\TAT)$), and let $r_A$ (resp. $R_A$) be $\rank(\K_1(\OAT))=\rank(\K_0(\OAT))$ (resp. $\K_1(\TAT)$).
Then, one has
\[\K_0(\OA)\cong \K_0(\OAT)=\Z^{r_A}\oplus T_A, \quad \K_1(\OA)\cong\K_1(\OAT)=\Z^{r_A},\]
and Lemma \ref{lem:DCOATAT}, \ref{lem:rankTA} show that
\[\Exts^1(\OA)= \K_0(\TAT)=\Z^{R_A+1}\oplus \tilde{T}_A,\]
\[\Exts^0(\OA)= \K_1(\TAT)=\Z^{R_A}.\]
Applying the same formula for the matrix $B$,
the statement follows immediately.

\end{proof}

We present the following theorem which is the main theorem in this paper.
\begin{theorem}\label{thm:mainthm}
Let $A=[A(i,j)]_{i,j=1}^N, B=[B(i,j)]_{i,j=1}^M$ 
be irreducible non permutation matrices with entries in $\{0,1\}.$
Then the following five conditions are mutually equivalent.
 \begin{enumerate}
\renewcommand{\theenumi}{(\roman{enumi})}
\renewcommand{\labelenumi}{\textup{\theenumi}}
\item
$\pi_i(\Aut(\OA))\cong \pi_i(\Aut(\OB)), i=1,2$.
\item $\OA \cong \OB.$
\item
$\K_0(\OA)\cong \K_0(\OB)$ and $\Exts^1(\OA) \cong \Exts^1(\OB)$.
\item
$\Extw^1(\OA)\cong \Extw^1(\OB)$ and $\Exts^1(\OA) \cong \Exts^1(\OB)$.
\item
$\Z^N/(I-A) \Z^N \cong \Z^M/(I-B) \Z^M$ and 
$\Z^N/(I-\widehat{A}) \Z^N \cong \Z^M/(I-\widehat{B}) \Z^M$.
\end{enumerate}
\end{theorem}
\begin{proof}
  The equivalence (i) $\Longleftrightarrow$ (ii) 
follows from Theorem \ref{thm:homotopy=isom}.
The implication (ii) $\Longrightarrow$ (iii)
is clear.
Since $\K_0(\OA) \cong \Extw^1(\OA) \cong \Z^N/(I-A) \Z^N$,
the equivalences among 
(iii), (iv) and (v) are obvious.
By Lemma \ref{lem:K0Ext1=>K_1Ext0},
we know that the condition (iii) implies that 
$\K_i(\OA)\cong \K_i(\OB)$ and $\Exts^i(\OA) \cong \Exts^i(\OB)$ for $i=1,2$.
Hence Proposition \ref{prop:pi_1} shows $\pi_i(\Aut(\OA))\cong \pi_i(\Aut(\OB)), i=1,2$
which is the condition (i).
\end{proof}
For $m \in \Z$, take a unitary $u_m \in Q(H)$ whose Fredholm index is $m$.
Take a trivial extension $\tau: \OA\rightarrow Q(H)$ which means that  
there exists a $*$-homomorphism $\tau_0: \OA \rightarrow B(H)$ such that 
$\tau = \pi \circ \tau_0$, where $\pi: B(H) \rightarrow Q(H)$ 
is a natural quotient map.
Let us denote by 
$\sigma_m$ the extension $\Ad(u_m)\circ \tau: \OA \rightarrow Q(H)$.
Let $q_A: \Extw(\OA) \rightarrow \Exts(\OA)$ denotes the natural quotient map.
Then the map $\iota_A: m \in \Z \rightarrow [\sigma_m]_s \in \Exts(\OA)$
yields a homomorphism of groups such that 
the sequence 
\begin{equation}\label{eq:ZExtsExtw}
 \Z \overset{\iota_A}{\longrightarrow} \Exts(\OA) \overset{q_A}{\longrightarrow} \Extw(\OA)
\end{equation}
is exact at the middle so that we have
$ \Exts(\OA)/\iota_A(\Z) $ is isomorphic to $\Extw(\OA).$
The sequence \eqref{eq:ZExtsExtw} is rephrased as 
\begin{equation*}
 \Z \overset{\hat{\iota}_A}{\longrightarrow}
 \Z^N/(I-\widehat{A})\Z^N 
 \overset{\hat{q}_A}{\longrightarrow}
 \Z^N/(I-A)\Z^N
\end{equation*}
(\cite[Lemma 3.1]{MaJMAA2024}).
\begin{corollary}
The pair $(\Exts(\OA), \iota_A(1))$ of the group 
$\Exts(\OA)$ and the position
$\iota_A(1)$ in $\Exts(\OA)$ is a complete invariant of the isomorphism class of $\OA$.
It means that the class $[(I-A)[e_1]]$ of the vector 
$(I-A)[e_1] \in \Z^N$, where $ e_1 = [1,0,\dots,0]^t$, 
in the quotient group
$\Z^N/(I - \widehat{A})\Z^N$ is a complete invariant of the isomorphism class of $\OA$.
\end{corollary}

\subsection{$\pi_i(\Aut(\OA\otimes K(H)))$}
In this subsection, we will refer to
 $\pi_i(\Aut(\OA\otimes K(H))).$
\begin{lemma}\label{lem:stabilization1}
$\pi_1(\Aut(\OA\otimes K(H))) \cong \pi_2(\Aut(\OA\otimes K(H))).$
\end{lemma}
\begin{proof}
We write $\K_0(\OA)=\Z^{r_A}\oplus T_A$ with a finite abelian group $T_A$.
By Proposition \ref{prop:pi_1timesK}, 
it is direct to see that 
both $\pi_1(\Aut(\OA\otimes K(H)))$ and 
$\pi_2(\Aut(\OA\otimes K(H)))$
are isomorphic to the abelian group
\begin{equation}\label{eq:TApi1pi2}
(T_A \otimes T_A) \, \oplus \, (T_A \otimes \Z^{r_A}) \, \oplus \, (\Z^{r_A} \otimes T_A)
\, \oplus \, (\Z^{r_A} \otimes \Z^{r_A})
\, \oplus \, (\Z^{r_A} \otimes \Z^{r_A}).
\end{equation}
\end{proof}
\begin{remark}
One can also check that, for a Kirchberg algebra $\A$ with finitely generated K-groups, the isomorphism $\pi_1(\operatorname{Aut}(\A\otimes K(H)))\cong \pi_2(\operatorname{Aut}(\A\otimes K(H)))$ holds if and only if $\rank (\K_0(\A))=\rank (\K_1(\A))$ holds.
\end{remark}
We see that the group structure of the abelian group
of \eqref{eq:TApi1pi2} determins the torsion group $T_A$ and the free abelian group $\Z^{r_A}$
as in the following proposition.
\begin{proposition}\label{prop:stabilization2}
$\pi_1(\Aut(\OA\otimes K(H)))\cong \pi_1(\Aut(\OB\otimes K(H)))$
if and only if $\OA\otimes K(H) \cong \OB\otimes K(H).$
\end{proposition}
\begin{proof}
We show the only if part.
We use the same notation as in the above proof,
and there are integers $\{n_i\}_{i=1}^k$, $\{m_j\}_{j=1}^l$ such that
\[0<n_1\leq n_2\leq\dots\leq n_k,\quad n_1|n_2|\cdots |n_k,\]
\[0<m_1\leq m_2\leq\dots\leq m_l,\quad m_1|m_2|\cdots |m_l,\]
\[T_A=\Z/n_1\Z\oplus \Z/n_2\Z\oplus\cdots\oplus\Z/n_k\Z,\quad T_B=\Z/m_1\Z\oplus\Z/m_2\Z\oplus\cdots\oplus\Z/m_l\Z.\]
By the fundamental theorem of finite abelian groups,
these integers are uniquely determined for the groups $T_A, T_B$.
The assumption $\pi_1(\Aut(\OA\otimes K(H)))\cong \pi_1(\Aut(\OB\otimes K(H)))$ and (\ref{eq:TApi1pi2}) yield
\[r_A=r_B=:r,\]
\begin{equation}\label{star}
(T_A\otimes T_A)\oplus T_A^{2r}\cong (T_B\otimes T_B)\oplus T_B^{2r}.
\end{equation}
Compairing the following expressions
\begin{align*}
\bigoplus_{i=1}^k(\Z/n_i\Z)^{2(k+r-i)+1}=&(T_A\otimes T_A)\oplus T_A^{2r}\\
=&(T_B\otimes T_B)\oplus T_B^{2r}\\
=&\bigoplus_{j=1}^l(\Z/m_j\Z)^{2(l+r-j)+1},
\end{align*}
we have $n_k=m_l=: n$.
For the groups
\[T'_A:=\Z/n_1\Z\oplus \dots \oplus\Z/n_{k-1}\Z,\quad T'_B:=\Z/m_1\Z\oplus\dots\oplus\Z/m_{l-1}\Z,\]
one has $T'_A\otimes \Z/n\Z=T'_A, T'_B\otimes \Z/n\Z=T'_B$,
and (\ref{star}) implies
\[(T'_A\otimes T'_A)\oplus {T'_A}^{2(r+1)}\cong (T'_B\otimes T'_B)\oplus {T'_B}^{2(r+1)}.\]
Now we can inductively obtain
\[n_{k-1}=m_{l-1}, n_{k-2}=m_{l-2},\dots,\quad k=l,\]
and this proves $\K_0(\OA)\cong \K_0(\OB)$.
\end{proof}
\begin{remark}
The analogy of the above proposition for general stable Kirchberg algebras $\A, \B$ with finitely generated K-groups dose not hold.
Let $P_\infty$ be the stable Kirchberg algebra KK-equivalent to $S(=C_0(0, 1))$.
Then, two algebras $\A=\OA\otimes K(H)$ and $\B=P_\infty\otimes\OA\otimes K(H)$ have the same homotopy groups of their automorphism groups, but they are not isomorphic in general (see the case that $\OA$ is the Cuntz algebra $\ON, N\geq 3$).
\end{remark}

\section{The invariants $(\Extw(\OA), \Exts(\OA))$}
In this section, we will determine the possible range of the pair $(\Extw(\OA), \Exts(\OA))$. 
\begin{proposition}\label{prop:groups}
For a finitely generated abelian group $M$ 
and an element $e \in\Z\oplus M,$ 
there exists a Cuntz--Krieger algebra $\OA$ for an irreducible non permutation matrix $A$
with entries in $\{0,1\}$ such that 
\begin{equation*}
\Extw(\OA) \cong (\Z\oplus M)/ \Z e, \qquad \Exts(\OA) \cong \Z\oplus M.
\end{equation*}
\end{proposition}
Before going to the proof of Proposition \ref{prop:groups},
we recall the range of K-groups of the purely infinite simple Cuntz--Krieger algebras.
For a finite abelian group 
\[T:=\Z/n_1\Z\oplus\cdots\oplus\Z/n_k\Z,\quad n_i\geq 2,\]
and an integer $r\geq 0$,
one can find a matrix $A$ of size $N:=(r+\sum_{i=1}^k(1+n_i))+3$ satisfying
\[\Z^N/(I-A)\Z^N=\Z^r\oplus T,\quad \Ker (I-A)=\Z^r\]
as follows.
Let $N_i, i=1, 2, \dots, k$ denote the $(1+n_i)\times (1+n_i)$-matrix whose entries are all $1$ and let $I_r$ denote the identity matrix of size $r$.
Then, the matrix $A$ is defined by
\[A:=\left[
\begin{array}{ccccccc}
&&&0&0&1\\
&\text{\Huge D}&&\vdots&\vdots&\vdots\\
&&&0&0&1\\
1&\dots&1&0&0&1\\
0&\dots&0&0&1&1\\
0&\dots&0&1&1&1\\
\end{array}
\right],\quad D:= \left[
\begin{array}{cccc}
I_r&&&\\
&N_1&&\\
&&\ddots&\\
&&&N_k
\end{array}
\right],\]
and it is straightforward to check that
\[\Z^N/(I_N-A)\Z^N\cong \Z^r\oplus\bigoplus_{i=1}^k\Z^{(1+n_i)}/(I_{1+n_i}-N_i)\Z^{(1+n_i)}=\Z^r\oplus T,\]
\[\Ker (I_N-A)=\Z^r.\]
In the directed graph with $N$-vertecies determined by $A$,
every two vertecies are connected by a directed path passing through the $N$th vertex,
which implies that $A$ is an irreducible, non permutation matrix.
Combining this with \cite{CCK}, for an arbitrary element $e\in \Z^r\oplus T$,
one can find a purely infinite simple Cuntz--Krieger algebra $\OB$ (i.e., $\OB$ with an irreducible non permutation matrix $B$) satisfying
\[(\K_0(\OB), [1_\OB]_0, \K_1(\OB))\cong (\Z^r\oplus T, e, \Z^r).\]
\begin{proof}[{Proof of Proposition \ref{prop:groups}}]
First,
we show the statement when $e\in\Z\oplus M$ is a non torsion element.
In this case one has a short exact sequence $0\to\Z\to \Z\oplus M\to (\Z\oplus M)/\Z e\to 0$ of abelian groups that gives an element of $\operatorname{Ext}_\Z^1((\Z\oplus M)/\Z e, \Z)$.
Let $\OB$ be a purely infinite simple Cuntz--Krieger algebra with $\K_0(\OB)\cong (\Z\oplus M)/\Z e$.
Then, the following UCT
\[0\to \operatorname{Ext}_\Z^1((\Z\oplus M)/\Z e, \Z)\to \Extw(\OB)\to \operatorname{Hom}_\Z(\K_1(\OB), \Z)\to 0\]
implies that there is an essential unital extension $0\to K(H) \to E\to \OB\to 0$ whose six term exact sequence splits into the following short exact sequences
\[0\to \Z\to \K_0(E)\to (\Z\oplus M)/\Z e\to 0,\]
\[0\to \K_1(E)\to \K_1(\OB)\to 0\]
and the first sequence is equivalent to the extension $\Z\to \Z\oplus M\to (\Z\oplus M)/\Z e$.
Hence one has $\Z\oplus M \cong K_0(E)$,
and the strong K-theoretic duality (see \cite{MaJMAA2024, PennigSogabe}) implies that there is a unital Kirchberg algebra $\mathcal{B}$ satisfying
\[\Exts(\mathcal{B})=K_0(E),\quad \Extw(\mathcal{B})=\K_0(\OB),\quad \mathcal{B}\underset{\sqK}{\sim} D(S\OB)\underset{\sqK}{\sim}\OB.\]
Applying \cite{CCK} to $\mathcal{B}\otimes K(H)\cong \OB\otimes K(H)$,
$\mathcal{B}$ is isomorphic to a purely infinite simple Cuntz--Krieger algebra $\OA$ (i.e, $\OA$ with an irreducible non permutation matrix $A$).

Next, we show the statement when $e$ is a torsion element.
We write $M:=\Z^m\oplus \tilde{T},\; e:=(0, \tilde{t})\in M,\; (\Z\oplus M)/\Z e=\Z^{1+m}\oplus T$ where $m\geq 0$ is an integer and $\tilde{T}, T=\tilde{T}/\Z \tilde{t}$ are finite abelian groups.
By \cite[Cor. 3.10]{Sogabe2022},
there exists a non torsion element $d\in \Z\oplus T$ satisfying
\[\tilde{T}=(\Z\oplus T)/\Z d.\]
For a purely infinite simple Cuntz--Krieger algebra $\OA$ with \[(\K_0(\OA), [1_\OA]_0, K_1(\OA))\cong (\Z^m\oplus(\Z\oplus T), (0, d), \Z^{m+1}),\]
one has
\[\Extw(\OA)\cong \K_0(\OA)\cong (\Z\oplus M)/\Z e, \quad \K_1(C_\OA)\cong \Z^m\oplus (\Z\oplus T)/\Z d=\Z^m\oplus \tilde{T}.\]
We have the following six term exact sequence of $0\to S\OA\to C_\OA\to \mathbb{C}\to 0$
\[\xymatrix{
\K_1(\OA)\ar[r]&\K_0(C_\OA)\ar[r]&\K_0(\mathbb{C})\ar[d]\\
0\ar[u]&\K_1(C_\OA)\ar[l]&\K_0(\OA),\ar[l]
}\]
where the vertical map $\K_0(\mathbb{C})\to \K_0(\OA)$ is identified with \[\K_0(u_\OA) : \K_0(\mathbb{C})\ni [1]_0\mapsto [1_\OA]_0=(0, d)\in \K_0(\OA).\]
Since $d$ is a non torsion element,
the map $\K_0(\mathbb{C})\to \K_0(\OA)$ is injective,
and the above six term exact sequence shows $\Z^{m+1}=\K_1(\OA)\cong \K_0(C_\OA)$.
Thus, the UCT
\[0\to \operatorname{Ext}_\Z^1(\K_1(C_\OA), \Z)\to \Ext_s(\OA)\to \operatorname{Hom}_\Z(\K_0(C_\OA), \Z)\to 0\]
implies $\Exts (\OA)\cong \tilde{T}\oplus\Z^{m+1}=\Z\oplus M$.
\end{proof}
Since $\Exts^1(\OA) = \K_0(\TAT)$,
Lemma \ref{lem:rankTA} yields the following corollary.
\begin{corollary}\label{range}
\begin{align*}
& \{(\Extw(\OA), \Exts(\OA) ) \mid A: 
\text{irreducible non permutation matrix with entries in } \{0,1 \}\} \\
= & \{(G, \Z\oplus M) \mid M:  \text{finitely generated abelian group, }
G \cong (\Z\oplus M)/\Z e \text{ for some } e \in \Z\oplus M \}
\end{align*}
\end{corollary}
One may notice that complete invariant of $\OA$ appears in two different ways.
One is a pair of $G$ and its element $d\in G$ (i.e., $(\K_0(\OA), [1_\OA]_0)$ due to \cite{Ro}),
and the other is the pair of the groups $(G, \Z\oplus M)$ as in the above corollary.
We will explain the relationship between these two aspects via the reciprocality.
\begin{lemma}[{c.f. \cite[Corollary 3.10]{Sogabe2022}}]\label{reci}
Let $M$ be a finitely generated abelian group.
For any $e\in\Z\oplus M$,
there exists an element $\tilde{d}\in \Z\oplus M$ satisfying
\[M\cong (\Z\oplus M)/\langle e, \tilde{d}\rangle.\]
\end{lemma}
\begin{proof}
We write $M=\Z\oplus T$ for a finite abelian group $T$ and an integer $m\geq 0$ so that $\Z\oplus M=\Z\oplus(\Z^m\oplus T)$.
We may assume that $e$ is of the form $(n, (0, t)), n\in\Z$,
and \cite[Corollary 3.10]{Sogabe2022} gives an element $\tilde{d}\in \Z\oplus 0\oplus T$ satisfying 
\[(\Z\oplus 0\oplus T)/\langle e, \tilde{d}\rangle\cong T.\]
Thus, the elements $e, \tilde{d}$ satisfy $(\Z\oplus M)/\langle e, \tilde{d}\rangle\cong M$.
\end{proof}
For $d\in G$ and $d'\in G'$, we say that $(G, d)$ and $(G', d')$ are equivalent if there is an isomorphism $\theta : G\to G'$ with $\theta (d)=d'$,
and denote by $[G, d]$ the equivalence class.
We write \[\mathcal{K}:=\{[G, d] | G : \text{finitely generated abelian group} \}=\{[\K_0(\OA), [1_\OA]_0]\},\]
\[\mathcal{E} := \{(G, \Z\oplus M) | G\cong (\Z\oplus M)/\Z e\}/_{\sim \text{isom}}=\{(\Extw(\OA), \Exts(\OA))\}/_{\sim \text{isom}},\]
where, in the definition of $\mathcal{E}$, two pairs $(G, \Z\oplus M)$ and $(G', \Z\oplus M')$ are said to be isomorphic if $G\cong G'$ and $\Z\oplus M\cong \Z\oplus M'$ hold.
\begin{proposition}\label{rel}
Let $G, M$ be finitely generated abelian groups. 
\begin{enumerate}
\renewcommand{\theenumi}{(\roman{enumi})}
\renewcommand{\labelenumi}{\textup{\theenumi}}
\item
For any $[G, d]\in \mathcal{K}$, one has $(G, \Z\oplus (G/\Z d))\in \mathcal{E}$.\\
\item
For any $(G, \Z\oplus M)\in\mathcal{E}$, there exist two elements $e, \tilde{d}\in \Z\oplus M$ satisfying
\[G\cong (\Z\oplus M)/\Z e, \quad (\Z\oplus M)/\langle e, \tilde{d} \rangle\cong M,\]
and the class $[(\Z\oplus M)/\Z e, \tilde{d}+\Z e]\in \mathcal{K}$ dose not depend on the choice of $e, \tilde{d}$.\\
\item
The above two constructions yield $\mathcal{K}=\mathcal{E}$ by which $[\K_0(\OA), [1_\OA]_0]$ corresponds to $(\Extw(\OA), \Exts(\OA))$.
\end{enumerate}
\end{proposition}
\begin{proof}
First, we prove (i).
Applying Lemma \ref{reci} to $d\in \Z\oplus G$,
there exists an element $\tilde{e}\in \Z\oplus G$ satisfying
\[G=(\Z\oplus G)/\langle d, \tilde{e}\rangle =(\Z\oplus (G/\Z d))/ \langle (\tilde{e}+\Z d)\rangle, \]
and one has $(G, \Z\oplus (G/\Z d))\in\mathcal{E}$.

Next, we show (ii).
Existence of $e, \tilde{d}$ follows from the definition of $\mathcal{E}$ and Lemma \ref{reci}. 
Since $(\Z\oplus M)/\Z e\cong G$ and $((\Z\oplus M)/\Z e)/\langle \tilde{d}+\Z e\rangle\cong M$ hold for any choice of $e, \tilde{d}$,
\cite[Proposition 2.19]{Sogabe2022} proves the statement.

Finally,
we show (iii).
The statements (i) and (ii) provide two correspondences $\mathcal{K}\xrightarrow{(i)}\mathcal{E}$ and $\mathcal{E}\xrightarrow{(ii)}\mathcal{K}$, and it is straightforward to check that the composition $\mathcal{E}\xrightarrow{(ii)}\mathcal{K}\xrightarrow{(i)}\mathcal{E}$ is the identity.

We show that the composition $\mathcal{K}\xrightarrow{(i)}\mathcal{E}\xrightarrow{(ii)}\mathcal{K}$ is the identity.
For a given $[G, d]\in\mathcal{K}$, the correspondence $\mathcal{K}\xrightarrow{(i)}\mathcal{E}$ gives the pair $(G, \Z\oplus (G/\Z d))$,
and the correspondence $\mathcal{E}\xrightarrow{(ii)}\mathcal{K}$ sends the pair to $[G, d']$ satisfying $G/\Z d'\cong G/\Z d$.
Thus, 
\cite[Proposition 2.19]{Sogabe2022} shows $[G, d']=[G, d]$.

Since $\Extw(\OA)\cong \K_0(\OA)$,
the following computation implies  that $\mathcal{K}\xrightarrow{(i)}\mathcal{E}$ sends $[\K_0(\OA), [1_\OA]_0]$ to $(\Extw(\OA), \Exts(\OA))$:
\begin{align*}
\Exts(\OA)=&KK(C_\OA, \mathbb{C})\\
\cong &\operatorname{Ext}_\Z^1(\K_1(C_\OA), \Z)\oplus\operatorname{Hom}_\Z(\K_0(C_\OA), \Z)\\
\cong& \Z^{\rank (\K_0(C_\OA))}\oplus \operatorname{Tor}(\K_1(C_\OA))\\
=&\Z\oplus \Z^{\rank(\K_1(C_\OA))}\oplus\operatorname{Tor}(\K_1(C_\OA))\\
=&\Z\oplus \K_1(C_\OA)\\
\cong&\Z\oplus (\K_0(\OA)/\Z [1_\OA]_0).
\end{align*}
\end{proof}

\section{Examples}
{\bf 1.} $\ON$:
Let $\ON$ be the Cuntz algebra of order $1<N \in \N$ (see \cite{Cuntz77}). 
It was shown that 
$\K_0(\ON) = \Z/(1 -N)\Z, \K_1(\ON) = 0 $ in \cite{Cuntz81}
and also 
$\Exts^1(\ON)=\Z, \Exts^0(\ON) =0$ in \cite{PS}, \cite{PP}. 
By Proposition \ref{prop:pi_1} and Corollary \ref{cor:pi_1}, 
 we have 
\begin{align*}
\pi_1(\Aut(\ON)) 
\cong & (\Exts^1(\ON) \otimes \K_0(\OA)) \, \oplus \,  (\Exts^0(\ON) \otimes \K_1(\OA)) \\
\cong & \Z \otimes \Z/(1 -N)\Z \\
\cong &  \Z/(1 -N)\Z ,\\
\pi_2(\Aut(\ON)) 
\cong & (\Exts^1(\ON) \otimes \K_1(\OA)) \,  \oplus \, (\Exts^0(\ON) \otimes \K_0(\OA)) \\
      & \oplus \Tor(\Exts^1(\ON),\K_0(\OA)) \cong 0, \\
\pi_1(\Aut(\ON \otimes K(H)))      
\cong &
\pi_2(\Aut(\ON\otimes K(H)))       
\cong \Z/ (1-N)\Z      .
\end{align*}

\medskip

\noindent
{\bf 2.} $\ON\otimes M_k(\bbC)$:
We note that $\ON\otimes M_k(\bbC)$ is realized as a Cuntz--Krieger algebra
$\mathcal{O}_{A^{<k>}}$ for the $Nk\times Nk$ matrix $A^{<k>}$  defined by
\[A^{<k>}:=\left[
\begin{array}{ccc|c}
0_N      & \cdots      &0_N   & [N] \\ \hline
I_N      &        &   &0_N   \\ 
       & \ddots &   &\vdots   \\
       &        & I_N &0_N   \\
\end{array}
\right], \quad [N]:=\left[
\begin{array}{ccc}
1&\cdots&1\\
\vdots&\ddots&\vdots\\
1&\cdots& 1
\end{array}
\right],\]
where $[N], 0_N, I_N$ are $N\times N$ matrices.
Hence the formulas of Proposition \ref{prop:pi_1} and Corollary \ref{cor:pi_1} 
are applicable for $\ON\otimes M_k(\bbC)$.
Since $\K_1(\ON\otimes M_k(\bbC)) = \K_1(\ON) =0$ and hence $\Exts^0(\ON) =0$,
we have
$$
\pi_1(\Aut(\ON \otimes M_k(\bbC)))\cong \Exts^1(\ON \otimes M_k(\bbC)) \otimes \Z/(1-N)\Z.
$$
By Paschke and Salinas \cite{PS},
we know that $\Exts^1(\ON \otimes M_k(\bbC)) \cong \Z \oplus \Z/(N-1, k)\Z$, 
where $(N-1, k)$ is the greatest common diviser of $N-1$ and $k.$ 
Hence we have  
$$
\pi_1(\Aut(\ON \otimes M_k(\bbC)))\cong  \Z/(1-N)\Z \, \oplus \, \Z/(N-1, k)\Z.
$$
Similarly we have
$$
\pi_2(\Aut(\ON \otimes M_k(\bbC)))\cong  \Z/(k, N-1)\Z.
$$
As $(\ON \otimes M_k(\bbC))\otimes  K(H) \cong \ON \otimes K(H),$
we have 
$\pi_i (\Aut(\ON \otimes M_k(\bbC))\otimes K(H))) \cong \Z/(1-N)\Z $ for $i=1,2.$

\medskip

\noindent
{\bf 3.}
Let $$A=
\begin{bmatrix}
1 & 1 & 1 \\
1 & 1 & 1 \\
1 & 0 & 0   
\end{bmatrix},
\qquad 
B= A^t =
\begin{bmatrix}
1 & 1 & 1 \\
1 & 1 & 0 \\
1 & 1 & 0   
\end{bmatrix}.
$$
As in \cite{MaPre2021exts}, we know that 
$$
\K_0(\OA) \cong \K_0(\OB) \cong \Z/2\Z, \qquad
\K_1(\OA) = \K_1(\OB) = 
\Exts^0(\OA) = \Exts^0(\OB) =0,
$$
and
$$
\Exts^1(\OA) \cong \Z,\qquad
\Exts^1(\OB) \cong \Z \oplus \Z/2\Z,
$$
so that 
$$
\pi_1(\Aut(\OA)) \cong \Z/2\Z, \qquad \pi_1(\Aut(\OB)) \cong \Z/2\Z \oplus Z/2\Z.
$$
Hence we know that $\OA$ is not isomorphic to $\OB.$
 Similarly we have
$$
\pi_2(\Aut(\OA)) \cong 0, \qquad \pi_2(\Aut(\OB)) \cong \Z/2\Z
$$
 and
$$
\pi_i(\Aut(\OA\otimes K(H))) 
\cong \pi_i(\Aut(\OB\otimes K(H)))
\cong
\Z/2\Z,
\qquad i=1,2. 
$$
 
The results of the present paper will be generalized to a wider class of Kirchberg algebras in a forthcoming paper.

\medskip

{\it Acknowledgment:}
K. Matsumoto is supported by JSPS KAKENHI Grant Number 19K03537.
T. Sogabe is supported by Research Fellow of the Japan Society for the Promotion of Science.

\end{document}